\documentclass[11pt]{article}

\usepackage[T1]{fontenc} \usepackage[utf8]{inputenc} \usepackage{lmodern} \usepackage[english]{babel}

\usepackage{amsmath,amssymb,amsthm,mathtools} \usepackage{mathrsfs} \usepackage{bbm} \usepackage{stmaryrd}

\usepackage{graphicx} \usepackage{float} \usepackage{changepage} \usepackage[ labelsep=period, font=footnotesize, labelfont=bf ]{caption}

\usepackage{enumitem}

\usepackage[ margin=1in ]{geometry}

 \allowdisplaybreaks

\usepackage{comment}

\usepackage{xcolor}

\usepackage[ colorlinks=true, linkcolor=black, citecolor=black, urlcolor=black ]{hyperref}

\usepackage[nameinlink]{cleveref}

\theoremstyle{plain} \newtheorem{theorem}{Theorem}[section]   \newtheorem{conjecture}[theorem]{Conjecture}

\theoremstyle{definition} \newtheorem{definition}[theorem]{Definition}

\theoremstyle{remark} 

\DeclareMathAlphabet{\mathbbmsl}{U}{bbm}{m}{sl}

\newcommand{\Zn}{\mathbbmsl{Z}_n} \newcommand{\ii}{\mathrm{i}}

\title{\Large\bfseries A Search for Hadamard Matrices of Near-Williamson Type}

\author{ Hadi Kharaghani\thanks{Department of Mathematics and Computer Science, University of Lethbridge, Lethbridge, AB T1K 3M4, Canada. Email: \texttt{kharaghani@uleth.ca}} \and Ali Mohammadian\thanks{School of Mathematics, Institute for Research in Fundamental Sciences (IPM), Tehran, Iran. Email: \texttt{ali\_m@ipm.ir}} \and Behruz Tayfeh-Rezaie\thanks{School of Mathematics, Institute for Research in Fundamental Sciences (IPM), Tehran, Iran. Email: \texttt{tayfeh-r@ipm.ir}} \and Vlad Zaitsev\thanks{Department of Mathematics and Computer Science, University of Lethbridge, Lethbridge, AB T1K 3M4, Canada. Email: \texttt{vlad.zaitsev@uleth.ca}} }

\begin{document}

\maketitle

\begin{abstract}
We investigate a generalization of Williamson matrices, which we call \emph{near-Williamson matrices}. These consist of four circulant $\{\pm1\}$-matrices $A,B,C,D$ of odd order $n$ satisfying $AA^\top+BB^\top+CC^\top+DD^\top=4nI,$ where $B$, $C$, and $D$ are symmetric.

Using an exhaustive computer search, we classify all inequivalent near-Williamson matrices of odd orders $n\le 47$ and establish their existence for every odd order $n\le 69$. As applications, we obtain a new quaternary Hadamard matrix of order $94$ and construct the first known quaternary Hadamard matrices of orders $118$, $130$, $134$, $206$, and $266$.

 \vspace{2mm} \noindent\textbf{Keywords:} Back-circulant matrix; circulant matrix; quaternary Hadamard matrix; Williamson-type Hadamard matrix. \smallskip

\end{abstract}

\section{Introduction}

A Hadamard matrix of order $n$ is an $n\times n$ $\{\pm1\}$-matrix $H$ satisfying $HH^\top=nI .$ It is well known that the order of a Hadamard matrix must be $1$, $2$, or a multiple of $4$ \cite{pal}. The Hadamard conjecture asserts that a Hadamard matrix exists for every positive integer divisible by $4$.

A matrix $A=[a_{ij}]_{0\le i,j\le n-1}$ is \emph{circulant} if $a_{ij}=a_{0,j-i}$ for all $i,j$ and \emph{back-circulant} if $a_{ij}=a_{0,i+j}$ for all $i,j$, where the indices are taken modulo $n$. Both circulant and back-circulant matrices are uniquely determined by their first rows. If $A$ is circulant and symmetric, then $a_{0k}=a_{0,n-k}$ for all $k$. Any two circulant matrices commute. Every back-circulant matrix is symmetric. Moreover, if $A$ is circulant and $B$ is back-circulant, then $AB$ is back-circulant and $AB = BA^\top .$ In particular, every symmetric circulant matrix commutes with every back-circulant matrix.
We denote by $R$ the back-circulant matrix whose first row is $(0,\ldots,0,1)$.

In 1944 Williamson introduced the array \[ W=
\begin{bmatrix}
A & B & C & D\\
-B & A & D & -C\\
-C & -D & A & B\\
-D & C & -B & A
\end{bmatrix}
\] for the construction of Hadamard matrices.
More precisely, let $A,B,C,D$ be $n\times n$ matrices with entries from $\{\pm1\}$. If they satisfy the \emph{additivity condition}
\begin{equation}
AA^\top+BB^\top+CC^\top+DD^\top=4nI,
\label{additivity}
\end{equation}
and the pairwise \emph{amicability condition}
\begin{equation}
XY^\top=YX^\top
\qquad
\text{for all }
X,Y\in\{A,B,C,D\},
\label{amicability}
\end{equation}
then $W$ is a Hadamard matrix of order $4n$. Such a quadruple $\{A,B,C,D\}$ is called a set of \emph{Williamson-type matrices} of order $n$.

Williamson \cite{w1} focused on the case where all four matrices are symmetric circulant matrices. A quadruple $\{A,B,C,D\}$ of symmetric circulant $\{\pm1\}$-matrices satisfying \eqref{additivity} is called a set of \emph{Williamson matrices}. Since symmetric circulant matrices commute, condition \eqref{amicability} is automatically satisfied. Several related plug-in constructions for Hadamard matrices have also been developed and investigated; see, for example, \cite{seb}. The Williamson construction was first implemented by computer search by Baumert, Golomb, and Hall \cite{H92}, who discovered a Hadamard matrix of order $92$. At that time, order $92$ was the smallest order for which a Hadamard matrix had not been constructed using the classical methods of Sylvester \cite{syl} and Paley \cite{pal}.

For many years it was believed that Williamson matrices might exist in every order. However, {\DJ}okovi\'c \cite{drag1} carried out an exhaustive search and found    that Williamson matrices do not exist in order $35$. Subsequently, Holzmann, Kharaghani, and Tayfeh-Rezaie \cite{Tay59} showed that Williamson matrices also do not exist in orders $47$, $53$, and $59$. They classified all inequivalent Williamson matrices of odd orders $n\le 59$. The corresponding numbers are shown in Table~\ref{xwill}. It is worth noting that the only known infinite family of Williamson matrices occurs in orders $\frac{q+1}{2},$ where $q$ is a prime power satisfying $q\equiv1\pmod4$ \cite{Turyn}.


\begin{table}[htbp]
\centering
\begin{tabular}{cc@{\qquad}cc}
\hline
$n$ & \texttt{\#} & $n$ & \texttt{\#} \\
\hline
 1 & 1 & 31 & 2 \\
 3 & 1 & 33 & 5 \\
 5 & 1 & 35 & 0 \\
 7 & 2 & 37 & 4 \\
 9 & 3 & 39 & 1 \\
11 & 1 & 41 & 1 \\
13 & 4 & 43 & 2 \\
15 & 4 & 45 & 1 \\
17 & 4 & 47 & 0 \\
19 & 6 & 49 & 1 \\
21 & 7 & 51 & 2 \\
23 & 1 & 53 & 0 \\
25 & 10 & 55 & 1 \\
27 & 6 & 57 & 1 \\
29 & 1 & 59 & 0 \\
\hline
\end{tabular}
\caption{The number of inequivalent Williamson matrices of odd orders $n \leqslant 59$ \cite{Tay59}.}\label{xwill}
\end{table}

Another important class of Williamson-type matrices has been studied extensively in the literature. Four circulant $\{\pm1\}$-matrices $A,B,C,D$ of order $n$ are called \emph{good matrices} if they satisfy \eqref{additivity}, the matrix $A-I$ is skew-symmetric, and $B,C,D$ are symmetric. In this situation, $AR,\; B,\; C,\; D$ satisfy \eqref{amicability}, and therefore form a set of Williamson-type matrices. Good matrices were introduced by Wallis in 1970 and were used to construct a skew-Hadamard matrix of order $92$ \cite{skew92}, which was then the smallest unresolved order for skew-Hadamard matrices.

As in the Williamson case, only a relatively small number of inequivalent examples exist in each admissible order. See \cite{GOOD,G-S,bad39,27-13} for details. In this paper we study a broader class that contains both Williamson matrices and good matrices.

\begin{definition}
Four circulant $\{\pm1\}$-matrices $A,B,C,D$ of order $n$ are called \emph{near-Williamson matrices} if they satisfy \eqref{additivity} and the matrices $B$, $C$, and $D$ are symmetric.
\end{definition}

Since $B$, $C$, and $D$ are symmetric, $AR,\; B,\; C,\; D$ satisfy \eqref{amicability}; hence $\{AR,B,C,D\}$ is a set of Williamson-type matrices. Near-Williamson matrices were introduced by Dutta, Mahanti, and Singh \cite{dutt}. Their exhaustive search in orders up to $17$ indicated that near-Williamson matrices are substantially more abundant than Williamson matrices. The main objective of this paper is to investigate the existence and classification of near-Williamson matrices in larger orders. In Section~\ref{result} we classify all inequivalent near-Williamson matrices of odd orders $n\le 47$. We also establish the existence of near-Williamson matrices for every odd order $n\le 69$. As an application, Section~\ref{complexxx} presents a new quaternary Hadamard matrix of order $94$ together with the first known examples of quaternary Hadamard matrices of orders $118$, $130$, $134$, $206$, and $266$. 
A very efficient search algorithm and computational techniques used in this work are described in Section~\ref{alg}.

\section{Near-Williamson Matrices}
\label{result}

We begin by defining the equivalence relation used throughout the paper. Let $X=[x_{ij}]_{0\le i,j\le n-1}$ be an $n\times n$ matrix and let $\sigma$ be an automorphism of the additive group $\Zn=\{0,1,\ldots,n-1\}.$ Define $X^\sigma= [x_{\sigma(i)\sigma(j)}]_{0\le i,j\le n-1}.$ It is straightforward to verify that the image of a circulant matrix under an automorphism of $\Zn$ is again circulant. A \emph{signed permutation matrix} is a matrix with entries from $\{-1,0,1\}$ having exactly one nonzero entry in every row and every column.

Two quadruples $\{A,B,C,D\} \qquad\text{and}\qquad \{A',B',C',D'\}$ of circulant $\{\pm1\}$-matrices of order $n$ are said to be \emph{equivalent} if, after a suitable permutation of $A',B',C',D'$, there exist

\begin{itemize}
\item an automorphism $\sigma$ of $\Zn$, and \item circulant signed permutation matrices $P,Q,S,T$
\end{itemize}
such that \[ A'=A^\sigma P, \qquad B'=B^\sigma Q, \qquad C'=C^\sigma S, \qquad D'=D^\sigma T. \]

\subsection{Exhaustive Search}

Using the algorithm described in Section~\ref{alg}, we performed an exhaustive computer search for near-Williamson matrices of odd orders $n\le 47.$ The search produced a complete classification of the equivalence classes in this range. The numbers of inequivalent near-Williamson matrices are displayed in Table~\ref{xnear}. The full database is available electronically at \cite{tayhome, zait}. The total computation time for all orders was approximately 450 hours on DRAC 192-core systems, corresponding to roughly ten CPU-core-years. The results for orders up to 37 were independently verified by two of the authors, using two different programs implemented and run separately on two different computers.

A striking feature of Table~\ref{xnear} is the rapid growth in the number of equivalence classes as the order increases. In contrast with Williamson matrices, near-Williamson matrices appear to be comparatively abundant. The available data suggest the following natural conjecture.

\begin{conjecture}
Near-Williamson matrices exist for every odd order.
\end{conjecture}

\begin{table}[htbp]
\centering
\begin{tabular}{rr@{\qquad}rr}
\hline
$n$ & \texttt{\#} & $n$ & \texttt{\#} \\
\hline
 1 & 1 & 25 & 2843 \\
 3 & 1 & 27 & 6579 \\
 5 & 1 & 29 & 6744 \\
 7 & 3 & 31 & 11837 \\
 9 & 5 & 33 & 57003 \\
11 & 5 & 35 & 73630 \\
13 & 24 & 37 & 56508 \\
15 & 96 & 39 & 206424 \\
17 & 96 & 41 & 170197 \\
19 & 200 & 43 & 274814 \\
21 & 1004 & 45 & 1406536 \\
23 & 827 & 47 & 654041 \\
\hline
\end{tabular}
\caption{Number of inequivalent near-Williamson matrices of odd orders $n \le 47$.}
\label{xnear}
\end{table}

\subsection{Good Matrices}

Good matrices form a distinguished subclass of near-Williamson matrices. Indeed, a quadruple $\{A,B,C,D\}$ of near-Williamson matrices is a set of good matrices precisely when $A-I$ is skew-symmetric. It was shown \cite{G-S,GOOD} that for odd orders $n\le 51$, good matrices exist if and only if $n\notin \{41,47,49,51\}.$ Using the exhaustive database of near-Williamson matrices obtained in the previous subsection, we examined every equivalence class of odd order $n\le 47$ to determine whether it contains a representative for which $A-I$ is skew-symmetric.

The resulting numbers are presented in Table~\ref{xgood}. These data provide an independent verification of the known existence results for good matrices obtained in \cite{GOOD,G-S,bad39,27-13}.

\begin{table}[htbp]
\centering
\begin{tabular}{cc@{\qquad}cc}
\hline
$n$ & \texttt{\#} & $n$ & \texttt{\#} \\
\hline
 1 & 1 & 25 & 9 \\
 3 & 1 & 27 & 13 \\
 5 & 1 & 29 & 5 \\
 7 & 3 & 31 & 3 \\
 9 & 1 & 33 & 15 \\
11 & 3 & 35 & 6 \\
13 & 6 & 37 & 2 \\
15 & 11 & 39 & 5 \\
17 & 2 & 41 & 0 \\
19 & 8 & 43 & 1 \\
21 & 10 & 45 & 4 \\
23 & 6 & 47 & 0 \\
\hline
\end{tabular}
\caption{Number of equivalence classes of near-Williamson matrices of odd orders $n \le 47$ that contain a set of good matrices.}
\label{xgood}
\end{table}

\subsection{Existence Search}

For orders in which Williamson matrices are either known not to exist or have not yet been found, namely $n\in\{53,59,65,67\},$ we carried out additional searches for near-Williamson matrices. For orders $53$ and $59$, we imposed the condition that the circulant matrix $A=[a_{ij}]_{0\le i,j\le n-1}$ be \emph{almost symmetric}, meaning that $a_{0k}=a_{0,n-k}, \qquad k=2,\ldots,n-2.$ For orders $65$ and $67$, examples were found using an improved search technique described in Section~\ref{improvement}. Together with results of Lang and Schneider \cite{LS}, these computations establish the existence of near-Williamson matrices for every odd order $n\le 69.$

Explicit examples in orders $53$, $59$, $65$, and $67$ are given in Tables~\ref{nearW53}--\ref{nearW67}. In those tables, $B$, $C$, and $D$ are symmetric, and the symbol ``$-$'' denotes the value $-1$. Our computations also produced examples of near-Williamson matrices in orders $73,\; 79,\; 91,\; 93,\; 95,\; 97,\; 103,\; 109,\; 127,\; 133,\; 181,$ which are available electronically at \cite{zait}. Recall that every set of Williamson matrices and every set of good matrices is, by definition, a set of near-Williamson matrices. Moreover, Turyn's infinite family \cite{Turyn} provides examples in every order $\frac{q+1}{2},$ where $q$ is a prime power satisfying $q\equiv1\pmod4.$ This accounts for the orders $79$, $91$, $97$, and $181$.  
In addition, good matrices of order $127$ were constructed by {\DJ}okovi\'c in 1993 \cite{drag2}. Consequently, these five orders were already known prior to the present work. However, the five orders we present are strictly near-Williamson, and are not equivalent to the known Williamson or good matrices.

\begin{table}[ht]
\begin{adjustwidth}{-1.5cm}{-1.5cm}
\centering
\scriptsize
\begin{tabular}{|l|l|}
\hline
$A$ &
$ 1 \, 1 \, - \, - \, - \, 1 \, 1 \, - \, 1 \, 1 \, - \, - \, - \, - \,
1 \, - \, 1 \, - \, 1 \, - \, 1 \, - \, - \, 1 \, 1 \, 1 \, 1 \, 1 \,
1 \, 1 \, 1 \, - \, - \, 1 \, - \, 1 \, - \, 1 \, - \, 1 \, - \, - \,
- \, - \, 1 \, 1 \, - \, 1 \, 1 \, - \, - \, - \, - $
\\
\hline
$B$ &
$ 1 \, - \, - \, - \, 1 \, 1 \, - \, 1 \, - \, 1 \, - \, - \, 1 \,
1 \, - \, 1 \, - \, - \, 1 \, 1 \, 1 \, 1 \, 1 \, 1 \, - \, 1 \,
1 \, 1 \, 1 \, - \, 1 \, 1 \, 1 \, 1 \, 1 \, 1 \, - \, - \, 1 \,
- \, 1 \, 1 \, - \, - \, 1 \, - \, 1 \, - \, 1 \, 1 \, - \, - \, - $
\\
\hline
$C$ &
$ 1 \, - \, 1 \, 1 \, 1 \, 1 \, - \, - \, 1 \, - \, - \, - \, 1 \,
- \, - \, - \, - \, - \, 1 \, 1 \, 1 \, 1 \, - \, 1 \, 1 \, - \,
- \, - \, - \, 1 \, 1 \, - \, 1 \, 1 \, 1 \, 1 \, - \, - \, - \,
- \, - \, 1 \, - \, - \, - \, 1 \, - \, - \, 1 \, 1 \, 1 \, 1 \, - $
\\
\hline
$D$ &
$ 1 \, - \, - \, - \, 1 \, - \, 1 \, 1 \, - \, - \, - \, 1 \, - \,
1 \, 1 \, - \, - \, - \, 1 \, - \, - \, - \, - \, 1 \, 1 \, - \,
1 \, 1 \, - \, 1 \, 1 \, - \, - \, - \, - \, 1 \, - \, - \, - \,
1 \, 1 \, - \, 1 \, - \, - \, - \, 1 \, 1 \, - \, 1 \, - \, - \, - $
\\
\hline
\end{tabular}
\caption{A quadruple of near-Williamson matrices of order $53$.}
\label{nearW53}
\end{adjustwidth}
\end{table}

\begin{table}[ht]
\begin{adjustwidth}{-1.5cm}{-1.5cm}
\centering
\scriptsize
\begin{tabular}{|l|l|}
\hline
$A$ &
$ 1 \, 1 \, - \, - \, - \, - \, 1 \, 1 \, 1 \, 1 \, - \, - \, - \, 1 \,
1 \, - \, - \, 1 \, - \, - \, - \, 1 \, - \, 1 \, 1 \, - \, 1 \, 1 \,
1 \, - \, - \, 1 \, 1 \, 1 \, - \, 1 \, 1 \, - \, 1 \, - \, - \, - \,
1 \, - \, - \, 1 \, 1 \, - \, - \, - \, 1 \, 1 \, 1 \, 1 \, - \, - \,
- \, - \, - $
\\
\hline
$B$ &
$ 1 \, - \, - \, - \, - \, 1 \, 1 \, - \, 1 \, 1 \, 1 \, 1 \, - \, 1 \,
- \, 1 \, 1 \, 1 \, - \, 1 \, - \, 1 \, 1 \, - \, 1 \, 1 \, 1 \, - \,
1 \, - \, - \, 1 \, - \, 1 \, 1 \, 1 \, - \, 1 \, 1 \, - \, 1 \, - \,
1 \, 1 \, 1 \, - \, 1 \, - \, 1 \, 1 \, 1 \, 1 \, - \, 1 \, 1 \, - \,
- \, - \, - $
\\
\hline
$C$ &
$ 1 \, - \, - \, 1 \, - \, - \, - \, - \, 1 \, 1 \, 1 \, 1 \, - \, 1 \,
- \, 1 \, - \, 1 \, - \, - \, 1 \, - \, - \, 1 \, - \, - \, - \, 1 \,
1 \, - \, - \, 1 \, 1 \, - \, - \, - \, 1 \, - \, - \, 1 \, - \, - \,
1 \, - \, 1 \, - \, 1 \, - \, 1 \, 1 \, 1 \, 1 \, - \, - \, - \, - \,
1 \, - \, - $
\\
\hline
$D$ &
$ 1 \, - \, - \, - \, - \, - \, 1 \, - \, - \, - \, 1 \, - \, 1 \, - \,
- \, 1 \, 1 \, 1 \, 1 \, 1 \, 1 \, - \, - \, 1 \, - \, - \, - \, 1 \,
1 \, 1 \, 1 \, 1 \, 1 \, - \, - \, - \, 1 \, - \, - \, 1 \, 1 \, 1 \,
1 \, 1 \, 1 \, - \, - \, 1 \, - \, 1 \, - \, - \, - \, 1 \, - \, - \,
- \, - \, - $
\\
\hline
\end{tabular}
\caption{A quadruple of near-Williamson matrices of order $59$.}
\label{nearW59}
\end{adjustwidth}
\end{table}

\begin{table}[ht]
\begin{adjustwidth}{-2cm}{-2cm}
\centering
\scriptsize
\begin{tabular}{|l|l|}
\hline
$A$ &
$ 1 \, 1 \, - \, 1 \, - \, - \, - \, - \, 1 \, 1 \, - \, - \, 1 \, 1 \,
1 \, - \, 1 \, 1 \, - \, - \, - \, - \, 1 \, - \, 1 \, - \, - \, 1 \,
- \, 1 \, - \, - \, - \, 1 \, 1 \, 1 \, 1 \, - \, - \, 1 \, 1 \, - \,
1 \, - \, - \, - \, - \, - \, 1 \, - \, - \, 1 \, 1 \, 1 \, - \, 1 \,
1 \, - \, - \, 1 \, - \, 1 \, 1 \, 1 \, - $
\\
\hline
$B$ &
$ 1 \, 1 \, - \, - \, 1 \, - \, 1 \, 1 \, - \, - \, 1 \, - \, - \, - \,
- \, - \, 1 \, - \, 1 \, - \, - \, - \, 1 \, 1 \, 1 \, 1 \, 1 \, 1 \,
1 \, - \, 1 \, 1 \, - \, - \, 1 \, 1 \, - \, 1 \, 1 \, 1 \, 1 \, 1 \,
1 \, 1 \, - \, - \, - \, 1 \, - \, 1 \, - \, - \, - \, - \, - \, 1 \,
- \, - \, 1 \, 1 \, - \, 1 \, - \, - \, 1 $
\\
\hline
$C$ &
$ 1 \, 1 \, - \, 1 \, 1 \, 1 \, - \, 1 \, - \, - \, 1 \, 1 \, 1 \, - \,
- \, 1 \, 1 \, 1 \, 1 \, 1 \, 1 \, 1 \, - \, - \, - \, 1 \, - \, - \,
1 \, - \, 1 \, - \, - \, - \, - \, 1 \, - \, 1 \, - \, - \, 1 \, - \,
- \, - \, 1 \, 1 \, 1 \, 1 \, 1 \, 1 \, 1 \, - \, - \, 1 \, 1 \, 1 \,
- \, - \, 1 \, - \, 1 \, 1 \, 1 \, - \, 1 $
\\
\hline
$D$ &
$ 1 \, - \, 1 \, - \, - \, 1 \, 1 \, - \, 1 \, 1 \, - \, 1 \, - \, 1 \,
1 \, 1 \, - \, - \, - \, 1 \, 1 \, 1 \, 1 \, 1 \, 1 \, - \, - \, 1 \,
- \, 1 \, - \, 1 \, 1 \, 1 \, 1 \, - \, 1 \, - \, 1 \, - \, - \, 1 \,
1 \, 1 \, 1 \, 1 \, 1 \, - \, - \, - \, 1 \, 1 \, 1 \, - \, 1 \, - \,
1 \, 1 \, - \, 1 \, 1 \, - \, - \, 1 \, - $
\\
\hline
\end{tabular}
\caption{A quadruple of near-Williamson matrices of order $65$.}
\label{nearW65}
\end{adjustwidth}
\end{table}

\begin{table}[ht]
\begin{adjustwidth}{-2cm}{-2cm}
\centering
\scriptsize
\begin{tabular}{|l|l|}
\hline
$A$ &
$ 1 \, - \, 1 \, 1 \, - \, 1 \, - \, 1 \, 1 \, - \, - \, 1 \, - \, - \,
- \, 1 \, - \, - \, 1 \, 1 \, 1 \, - \, - \, - \, - \, 1 \, - \, 1 \,
1 \, - \, - \, 1 \, - \, 1 \, 1 \, - \, 1 \, - \, - \, 1 \, - \, - \,
- \, - \, 1 \, - \, 1 \, - \, 1 \, - \, - \, 1 \, 1 \, 1 \, 1 \, 1 \,
- \, - \, 1 \, 1 \, - \, 1 \, - \, 1 \, - \, - \, - $
\\
\hline
$B$ &
$ 1 \, - \, - \, 1 \, - \, - \, 1 \, 1 \, 1 \, - \, - \, 1 \, 1 \, 1 \,
1 \, 1 \, 1 \, - \, 1 \, 1 \, - \, - \, - \, - \, - \, 1 \, - \, - \,
- \, 1 \, - \, - \, - \, - \, - \, - \, - \, - \, 1 \, - \, - \, - \,
1 \, - \, - \, - \, - \, - \, 1 \, 1 \, - \, 1 \, 1 \, 1 \, 1 \, 1 \,
1 \, - \, - \, 1 \, 1 \, 1 \, - \, - \, 1 \, - \, - $
\\
\hline
$C$ &
$ 1 \, 1 \, - \, - \, 1 \, 1 \, - \, - \, - \, 1 \, - \, 1 \, 1 \, 1 \,
- \, - \, - \, - \, 1 \, 1 \, - \, 1 \, - \, 1 \, - \, 1 \, - \, - \,
1 \, - \, - \, - \, - \, 1 \, 1 \, - \, - \, - \, - \, 1 \, - \, - \,
1 \, - \, 1 \, - \, 1 \, - \, 1 \, 1 \, - \, - \, - \, - \, 1 \, 1 \,
1 \, - \, 1 \, - \, - \, - \, 1 \, 1 \, - \, - \, 1 $
\\
\hline
$D$ &
$ 1 \, - \, 1 \, - \, 1 \, 1 \, - \, - \, - \, - \, - \, - \, 1 \, 1 \,
1 \, 1 \, 1 \, - \, - \, - \, 1 \, - \, - \, - \, - \, 1 \, - \, 1 \,
- \, - \, 1 \, 1 \, - \, 1 \, 1 \, - \, 1 \, 1 \, - \, - \, 1 \, - \,
1 \, - \, - \, - \, - \, 1 \, - \, - \, - \, 1 \, 1 \, 1 \, 1 \, 1 \,
- \, - \, - \, - \, - \, - \, 1 \, 1 \, - \, 1 \, - $
\\
\hline
\end{tabular}
\caption{A quadruple of near-Williamson matrices of order $67$.}
\label{nearW67}
\end{adjustwidth}
\end{table}

\section{New Quaternary Hadamard Matrices}
\label{complexxx}

A \emph{quaternary Hadamard matrix} of order $n$ is an $n\times n$ matrix $H$ with entries from $\{-\ii,-1,1,\ii\}$ satisfying $HH^*=nI,$ where $\ii$ denotes the imaginary unit and $H^*$ denotes the conjugate transpose of $H$. It is well known that the order of a quaternary Hadamard matrix must be either $1$ or an even integer \cite{Sebe}. Whether a quaternary Hadamard matrix exists for every even order remains an open problem. Following the recent construction by Sz\"oll\H{o}si of a quaternary Hadamard matrix of order $94$ \cite{Sz}, the smallest order for which the existence of a quaternary Hadamard matrix appears to be unresolved is $118$; see Table 11.2 of \cite{94}.

Suppose that four $n\times n$ $\{\pm1\}$-matrices $A,B,C,D$ satisfy both the additivity condition \eqref{additivity} and the amicability condition \eqref{amicability}. It is straightforward to verify that the block matrix

\begin{equation}
\label{complex}
\frac{1+\ii}{2}
\begin{bmatrix}
A+\ii B & C+\ii D \\
\ii C + D & -\ii A - B
\end{bmatrix}
\end{equation}
is a quaternary Hadamard matrix of order $2n$ \cite{Tur}. Applying construction \eqref{complex} to the newly discovered near-Williamson matrices yields several new quaternary Hadamard matrices. Specifically, substituting $\{AR,B,C,D\}$ for the near-Williamson matrices of orders $59,\quad 65,\quad 67,\quad 103,\quad 133$ produces new quaternary Hadamard matrices of orders $118,\quad 130,\quad 134,\quad 206,\quad 266,$ respectively. To the best of our knowledge, these are the first known examples of quaternary Hadamard matrices in those orders.

Likewise, applying the same construction to the new near-Williamson matrix of order $47$ yields a new quaternary Hadamard matrix of order $94.$ This provides an alternative construction of a quaternary Hadamard matrix in that order.

\section{The Algorithm}
\label{alg}

In this section we describe the search algorithm used to find near-Williamson matrices of odd orders. The overall strategy is similar to the algorithm employed in \cite{Tay59} for the classification of Williamson matrices. Throughout this section, let $\{A,B,C,D\}$ be a quadruple of near-Williamson matrices of odd order $n$, with $B$, $C$, and $D$ symmetric circulant matrices.

\subsection{Spectral Conditions}

Since the matrices $AA^\top,\quad BB^\top,\quad CC^\top,\quad DD^\top$ are all circulant, they commute pairwise. Hence they possess a common basis of eigenvectors. For each matrix $X\in\{A,B,C,D\}$, let $\lambda_0(X),\lambda_1(X),\ldots,\lambda_{n-1}(X)$ denote the eigenvalues of $XX^\top$ with respect to this common eigenbasis. By the additivity condition \eqref{additivity}, we obtain
\begin{equation}
\lambda_i(A)+
\lambda_i(B)+
\lambda_i(C)+
\lambda_i(D)
=
4n,
\qquad
i=0,1,\ldots,n-1.
\label{eigen}
\end{equation}

Recall that if $X$ is a circulant matrix with first row $(x_0,x_1,\ldots,x_{n-1}),$ then its eigenvalues are $\mu_i = \sum_{k=0}^{n-1} x_k\omega_i^k, \qquad i=0,1,\ldots,n-1,$ where $\omega_i=e^{2\pi \ii i/n}$ are the $n$th roots of unity. From this representation, one obtains an expression for the eigenvalues of $XX^\top$ in terms of row inner products. Let $p_k(X)$ denote the inner product of row $0$ and row $k$ of the circulant matrix $X$. Then
\begin{equation}
\lambda_i(X)
=
\sum_{k=0}^{n-1}
p_k(X)
\cos\!\left(
\frac{2ik\pi}{n}
\right),
\qquad
i=0,1,\ldots,n-1.
\label{eigenmu}
\end{equation}

The relations \eqref{eigen} and \eqref{eigenmu} provide powerful constraints that substantially reduce the search space. Since all eigenvalues of $AA^\top$, $BB^\top$, $CC^\top$, and $DD^\top$ are nonnegative, equation \eqref{eigen} implies
\begin{align}
\lambda_i(A)
&\le 4n,
\label{eigenA}
\\
\lambda_i(B)
&\le 4n,
\label{eigenB}
\end{align}
and
\begin{equation}
\lambda_i(A)+\lambda_i(B)
\le 4n,
\qquad
i=0,1,\ldots,n-1.
\label{eigenAB}
\end{equation}
These inequalities eliminate large numbers of potential candidates during the search.

The most useful instance of \eqref{eigen} occurs for $i=0$. Using \eqref{eigenmu}, we obtain
\begin{equation}
a^2+b^2+c^2+d^2
=
4n,
\label{foursq}
\end{equation}
where $a, b, c, d$ are the row sums of $A$, $B$, $C$, and $D$, respectively. Equation \eqref{foursq} significantly reduces the collection of admissible row sums and therefore dramatically decreases the number of candidate matrices.

By negating matrices if necessary, we may assume without loss of generality that the first entry of the first row of each matrix is $1$. Since $B$, $C$, and $D$ are symmetric circulant matrices, their row sums satisfy $b\equiv c\equiv d\equiv n \pmod 4.$ We also assume, without loss of generality, that $|b| \ge |c| \ge |d|.$

\subsection{Construction of Candidates}

We begin by constructing the set $\mathcal{M}$ of all symmetric circulant $\{\pm1\}$-matrices of order $n$ having row sum $b$ and first entry equal to $1$. Next, the automorphism group of $\Zn$ is applied to $\mathcal{M}$. From every orbit we select a representative matrix $X$ and test whether all eigenvalues of $XX^\top$ are bounded above by $4n$. Matrices satisfying this condition are retained, producing a reduced set $\mathcal{M}'.$ This set contains all possible candidates for $B$ and is stored in memory for later use.

We then consider every circulant $\{\pm1\}$-matrix $Y$ of order $n$ with row sum $a$ and first entry equal to $1$. If there exists a circulant signed permutation matrix $P$ such that $YP$ is lexicographically larger than $Y$, then $Y$ is discarded. Otherwise, we compute the eigenvalues of $YY^\top$. If any eigenvalue exceeds $4n$, then $Y$ is discarded. If all eigenvalues satisfy the required bound, then $Y$ becomes a candidate for $A$. We set $A=Y$ and continue the search. For each matrix $Z\in\mathcal{M}',$ we determine whether any eigenvalue of $AA^\top+ZZ^\top$ exceeds $4n$. Whenever no eigenvalue exceeds this bound, $Z$ becomes a valid candidate for $B$, and we set $B=Z.$ The remaining task is then to determine all admissible matrices $C$ and $D$.

\subsection{Determining the Matrices \texorpdfstring{$C$}{C} and
\texorpdfstring{$D$}{D}}

At this stage the search can be reduced substantially by exploiting the relations imposed by the additivity condition. To illustrate the idea, consider the case $n=7$. Let $(1,c_1,c_2,c_3,c_3,c_2,c_1)$ be the first row of $C$, and let $(1,d_1,d_2,d_3,d_3,d_2,d_1)$ be the first row of $D$. Using \eqref{additivity} and comparing the inner products of row $0$ with rows $1$, $2$, and $3$, we obtain
\begin{equation}
\label{cdeqns}
\left\{
\begin{aligned}
2c_1+2c_1c_2+1+2c_2c_3
&+
2d_1+2d_1d_2+1+2d_2d_3
=
p_1,
\\
2c_2+1+2c_1c_3+2c_2c_3
&+
2d_2+1+2d_1d_3+2d_2d_3
=
p_2,
\\
2c_3+2c_1c_2+2c_1c_3+1
&+
2d_3+2d_1d_2+2d_1d_3+1
=
p_3,
\end{aligned}
\right.
\end{equation}
where the integers $p_1,\ p_2,\ p_3$ depend only on the previously chosen matrices $A$ and $B$.

Dividing all equations in \eqref{cdeqns} by $2$ and using the fact that $xy \equiv x+y-1 \pmod 4, \qquad x,y\in\{-1,1\},$ we obtain a system of linear congruences modulo $4$:
\begin{equation}
\label{CDmod4}
\left\{
\begin{aligned}
c_3+d_3 &= q_1,
\\
c_1+d_1 &= q_2,
\\
c_2+d_2 &= q_3,
\end{aligned}
\right.
\end{equation}
where each $q_i$ depends only on the corresponding value $p_i$. The congruences \eqref{CDmod4} show that, once $A$ and $B$ are fixed, each variable $d_i$ is either equal to $c_i$ or to $-c_i$. More generally, the equation involving $q_k$ contains the unique sum $c_i+d_i$ corresponding to the index satisfying $2i\equiv \pm k \pmod n.$ Since $n$ is odd, this establishes a bijection between the equations and the indices. Consequently all signs are determined directly.

We therefore write $d_i=s_i c_i, \qquad s_i\in\{-1,1\},$ where the signs $s_i$ are known once $A$ and $B$ have been chosen. Substituting $d_i=s_i c_i$ into \eqref{cdeqns} yields
\begin{equation}
\label{cdeqnsother}
\left\{
\begin{aligned}
2(1+s_1)c_1
&+
2(1+s_1s_2)c_1c_2
+
2(1+s_2s_3)c_2c_3
=
p_1-2,
\\
2(1+s_2)c_2
&+
2(1+s_1s_3)c_1c_3
+
2(1+s_2s_3)c_2c_3
=
p_2-2,
\\
2(1+s_3)c_3
&+
2(1+s_1s_2)c_1c_2
+
2(1+s_1s_3)c_1c_3
=
p_3-2.
\end{aligned}
\right.
\end{equation}
Dividing each equation by $4$, we convert \eqref{cdeqnsother} into a linear system modulo $4$ whose variables are the entries of $C$. The resulting system is then transformed into standard reduced echelon form.

In practice, this procedure is surprisingly efficient. The number of resulting $\{\pm1\}$-solutions is typically very small, and in many cases a unique solution is obtained. An additional improvement comes from observing that the coefficients of \eqref{cdeqnsother} depend only on the sign pattern $(s_1,s_2,s_3,\ldots)$ and not on the specific choice of $A$ and $B$. The matrices $A$ and $B$ affect only the right-hand side. Since the same sign pattern arises repeatedly during the search, we compute the reduced echelon form only once for each sign pattern and store the result. Whenever the same pattern reappears, all solutions are obtained simply by substituting the new right-hand side into the stored echelon form. This caching strategy significantly reduces the overall computation time.

Finally, every solution produced by the linear system is substituted back into the original additivity condition $AA^\top+BB^\top+CC^\top+DD^\top=4nI$ to verify correctness. All valid solutions are retained as candidate near-Williamson quadruples.

\medskip \noindent \textbf{Algorithm for Finding All Near-Williamson Matrices of Odd Order $n$}

\begin{enumerate}[label=\textbf{A\arabic*.},leftmargin=*,labelsep=0.5em,%
topsep=4pt,itemsep=3pt,parsep=0pt]

\item Initialize $\mathcal{S}=\varnothing.$

\item Determine all integer solutions of $a^{2}+b^{2}+c^{2}+d^{2}=4n$ satisfying $|b|\ge |c|\ge |d|$ and $b\equiv c\equiv d\equiv n \pmod 4.$

\item For every solution \((a,b,c,d)\), perform the following steps.

\begin{enumerate}[label=\textbf{B\arabic*.},leftmargin=*,labelsep=0.5em,%
topsep=3pt,itemsep=3pt,parsep=0pt]

\item Construct the set \(\mathcal{M}\) consisting of all symmetric circulant \(\{\pm1\}\)-matrices of order \(n\) with row sum \(b\) and first entry equal to \(1\).

\item Apply the action of the automorphism group of \(\Zn\) to \(\mathcal{M}\). For every orbit select a representative \(X\). Retain only those representatives satisfying $\lambda_i(X)\le 4n, \qquad i=0,\ldots,n-1.$ Denote the resulting collection by $\mathcal{M}'.$

\item For every circulant \(\{\pm1\}\)-matrix \(Y\) of order \(n\) with row sum \(a\) and first entry equal to \(1\), perform the following steps.

\begin{enumerate}[label=\textbf{C\arabic*.},leftmargin=*,labelsep=0.5em,%
topsep=3pt,itemsep=3pt,parsep=0pt]

\item If there exists a circulant signed permutation matrix \(P\) such that \(YP\) is lexicographically larger than \(Y\), discard \(Y\). Otherwise proceed to the next step.

\item Compute the eigenvalues of \(YY^\top\). If any eigenvalue exceeds \(4n\), discard \(Y\). Otherwise set $A=Y$ and proceed.

\item For every matrix $Z\in\mathcal{M}',$ perform the following steps.

\begin{enumerate}[label=\textbf{D\arabic*.},leftmargin=*,labelsep=0.5em,%
topsep=3pt,itemsep=3pt,parsep=0pt]

\item Compute the eigenvalues of $AA^\top+ZZ^\top .$ If any eigenvalue exceeds \(4n\), discard \(Z\). Otherwise set $B=Z$ and proceed.

\item Determine the signs $s_i\in\{-1,1\}$ such that $d_i=s_i c_i$ using the congruence system analogous to \eqref{CDmod4}. The congruence involving \(q_k\) determines the sign corresponding to the index satisfying $2i\equiv\pm k \pmod n.$

\item Form the system of equations arising from $AA^\top+BB^\top+CC^\top+DD^\top=4nI,$ expressed solely in terms of the entries of \(C\), as in \eqref{cdeqnsother}. Divide all equations by \(4\).

\item Convert the resulting system into a linear system modulo \(4\). Retrieve the stored reduced echelon form corresponding to the sign pattern $(s_1,s_2,\ldots).$ If this sign pattern has not occurred previously, compute the reduced echelon form and store it for future use.

\item Substitute the current right-hand side into the stored reduced echelon form and determine all \(\{\pm1\}\)-solutions for \(C\).

\item For each candidate solution, construct $D=(s_i c_i)$ and verify the original equation $AA^\top+BB^\top+CC^\top+DD^\top=4nI.$ Add every valid quadruple $(A,B,C,D)$ to the set \(\mathcal{S}\).

\end{enumerate}

\end{enumerate}

\end{enumerate}

\item Check all elements of \(\mathcal{S}\) for equivalence and retain a representative from each equivalence class.

\end{enumerate}

\subsection{Using Multipliers of Order Three}
\label{improvement}

We now describe the computational improvement that enabled us to find examples of near-Williamson matrices in orders beyond $59$. Let $u$ be an element of multiplicative order $3$ in $\Zn$, and let $\sigma(k)=uk$ be the corresponding automorphism of $\Zn$. When $n=p$ is prime, such an element exists if and only if $p\equiv 1 \pmod 6.$ More generally, an element of order $3$ exists in $\Zn$ if and only if either $9\mid n$ or at least one prime divisor of $n$ is congruent to $1$ modulo $3$. The nonzero elements of $\Zn$ decompose into orbits of the form $\{k,uk,u^{2}k\},$ each orbit having size either $3$ or $1$.

We restrict our search to quadruples $\{A,B,C,D\}$ that are invariant under the action of $\sigma$. Up to relabelling of the matrices $B$, $C$, and $D$, this can occur in exactly two ways.

\begin{enumerate}[label=\textnormal{(\roman*)}]

\item Every matrix is invariant under $\sigma$: $X^\sigma=X, \qquad X\in\{A,B,C,D\}.$ In this case, the first rows of all four matrices are constant on the orbits of $\sigma$.

\item The matrices satisfy $A^\sigma=A, \qquad C=B^\sigma, \qquad D=C^\sigma.$ Hence the matrices $B$, $C$, and $D$ are obtained from a single symmetric circulant matrix via the action of $\sigma$. Since automorphisms preserve symmetry of circulant matrices, $C$ and $D$ are symmetric whenever $B$ is.

\end{enumerate}

The following theorem formalizes the key observation that makes this approach computationally effective.

\begin{theorem}
\label{thmorder3}
Let $\sigma(k)=uk,$ where $u$ has multiplicative order $3$ in $\Zn$. Suppose that circulant $\{\pm1\}$-matrices $A,B,C,D$ form a quadruple of one of the following two types:

\begin{enumerate}[label=\textnormal{(\roman*)}]

\item $X^\sigma=X \quad \text{for all } X\in\{A,B,C,D\};$

\item $A^\sigma=A, \qquad C=B^\sigma, \qquad D=C^\sigma .$

\end{enumerate}

Then the quantity $p_k(A)+p_k(B)+p_k(C)+p_k(D)$ is constant on every orbit of $\sigma$.

\end{theorem}

\begin{proof}
Let $X$ be a circulant matrix. If the first row of $X$ is $(x_0,x_1,\ldots,x_{n-1}),$ then the first row of $X^\sigma$ is given by $k\longmapsto x_{uk}.$ Consequently, $p_k(X^\sigma)=p_{uk}(X).$

Suppose first that condition (i) holds. Then $p_k(X)=p_{uk}(X)$ for every matrix $X\in\{A,B,C,D\}$. Therefore their sum is constant on every orbit of $\sigma$. Now suppose that condition (ii) holds. Since $A^\sigma=A,$ we have $p_k(A)=p_{uk}(A).$ Furthermore, $C=B^\sigma, \qquad D=C^\sigma,$ so $p_k(B)+p_k(C)+p_k(D) = p_k(B)+p_{uk}(B)+p_{u^2k}(B).$ Replacing $k$ by $uk$ merely permutes these three terms. Hence the sum remains unchanged. Therefore $p_k(A)+p_k(B)+p_k(C)+p_k(D)$ is constant on each orbit of $\sigma$.

\end{proof}

In either of the two situations above, the number of independent variables in the search decreases to roughly one third of its original value. By Theorem~\ref{thmorder3}, the number of correlation conditions that must be checked decreases by approximately the same factor. This reduction makes it possible to investigate orders far beyond the range accessible to the exhaustive search described earlier. The examples of near-Williamson matrices in orders $65 \qquad\text{and}\qquad 67$ were found using this method. The order-$65$ example is of type~(i) with $u=16,$ while the order-$67$ example is of type~(ii) with $u=29.$ The same approach was also used to obtain the examples in all the additional orders listed in Section~\ref{result}, with the exception of orders $53$ and $59$.

\bigskip

\section*{Acknowledgement}
This research was initiated during a visit by Behruz Tayfeh-Rezaie to
the University of Lethbridge. He expresses his sincere gratitude to
Hadi Kharaghani for the hospitality and support provided during his
visit. 
The research was made possible in part by the support provided by the Digital Research Alliance of Canada. Hadi Kharaghani and Vlad Zaitsev were supported by the Natural Sciences and Engineering Research Council of Canada (NSERC).

\end{document}